\documentclass[12pt]{article}%
\usepackage{graphicx}
\usepackage{amsmath}
\usepackage{amsfonts}
\usepackage{amssymb}
\usepackage{epstopdf}
\usepackage{todonotes}
\usepackage{graphicx,subcaption,caption,xcolor}
\usepackage[colorlinks,urlcolor=blue]{hyperref}%
\providecommand{\U}[1]{\protect\rule{.1in}{.1in}}

\providecommand{\U}[1]{\protect\rule{.1in}{.1in}}
\providecommand{\U}[1]{\protect\rule{.1in}{.1in}}
\newcommand{\Id}{\operatorname{Id}}

\newtheorem{theorem}{Theorem}

\newtheorem{algorithm}[theorem]{Algorithm}

\numberwithin{equation}{section}
\numberwithin{theorem}{section}
\newtheorem{definition}[theorem]{Definition}
\newtheorem{example}[theorem]{Example}

\newtheorem{lemma}[theorem]{Lemma}

\newtheorem{remark}[theorem]{Remark}

\newenvironment{proof}[1][Proof]{\textbf{#1.} }{\ \rule{0.5em}{0.5em}}
\begin{document}

\title{\textbf{The Cyclic Douglas-Rachford Algorithm with }$\boldsymbol{r}%
$\textbf{-sets-Douglas-Rachford Operators}}
\author{Francisco J. Arag\'{o}n Artacho$^{1}$, Yair Censor$^{2}$ and Aviv
Gibali$^{3,4}$\bigskip \\$^{1}$Department of Mathematics, University of Alicante, \\03690 San Vicente del Raspeig,
Alicante, Spain\\
\{francisco.aragon@ua.es\}\medskip \\
$^{2}$Department of Mathematics, University of Haifa\\
Mt.\ Carmel, Haifa 3498838, Israel. \{yair@math.haifa.ac.il\}\medskip\\
$^{3}$Department of Mathematics, ORT Braude College \\
Karmiel 2161002, Israel.\\
$^{4}$The Center for Mathematics and Scientific Computation\\
University of Haifa, Mt. Carmel\\
Haifa 3498838, Israel. \{avivg@braude.ac.il\}}
\date{January 1, 2018. Revised: May 20, 2018.\\
 Accepted for publication in Optimization Methods and Software (OMS) July 17, 2018. }
\maketitle

\begin{abstract}
The Douglas-Rachford (DR) algorithm is an iterative procedure that uses
sequential reflections onto convex sets and which has become popular for
convex feasibility problems. In this paper we propose a structural
generalization that allows to use $r$-sets-DR operators in a cyclic fashion.
We prove convergence and present numerical illustrations of the potential advantage of such operators with $r>2$ over the classical $2$-sets-DR operators in a cyclic algorithm.\medskip

\medskip


\end{abstract}

\section{Introduction\label{sec:Intro}}

We consider the convex feasibility problem (CFP) in a real Hilbert space
$\mathcal{H}$. For $i=0,1,\cdots,m-1,$ let $C_{i}\subseteq\mathcal{H}$ be
nonempty, closed and convex sets. The CFP is to%
\begin{equation}
\text{find a point }x^{\ast}\in C:=\cap_{i=0}^{m-1}C_{i}. \label{P:CFP}%
\end{equation}
The literature about projection methods for solving this problem is vast, see,
e.g., \cite{bb96}, \cite{cc-review}, \cite[Chapter 5]{CZ97} or the recent
\cite{BC-book}. The Douglas--Rachford (DR) algorithm whose origins are in
\cite{DR56} is a recent addition to this class of methods. We are unable to
compete with the excellent coverage of the literature on this algorithm
furnished in the recent 2017 paper by Bauschke and Moursi \cite{bm17} and
direct the reader there.
The DR algorithm has witnessed a surge of interest and publications investigating it in all directions, such as, e.g., for the non-convex and inconsistent case \cite{bdl18,benoist15, abt16}. A particular research direction consists of creating and studying new algorithmic structures that rely on the principles of the original DR algorithm.

This work belongs to this direction. We present and study a new algorithmic
structure for the DR algorithm that cyclically uses $r$-sets-DR operators. In
order to explain this recall the original $2$-sets-DR algorithm. Given two
sets $C_{0}$ and $C_{1}$ denote by $P_{C_{i}}$ the orthogonal projection onto
$C_{i}$ and denote the reflection with respect to $C_{i}$ by $\mathcal{R}%
_{C_{i}}=2P_{C_{i}}-\Id$, for $i=0,1$, where $\Id$ is the identity operator on
$\mathcal{H}$. With the combined operator $\mathcal{V}_{C_{0},C_{1}%
}:=\mathcal{R}_{C_{1}}\mathcal{R}_{C_{0}}$ the original\textit{ }%
$2$-sets-DR\textit{ }operator is defined as%
\begin{equation}\label{eq:DR_orig}
\mathcal{T}_{C_{0},C_{1}}:=\frac{1}{2}\left(  \Id+\mathcal{V}_{C_{0},C_{1}%
}\right)  .
\end{equation}
The original DR algorithm, starting from an arbitrary $x_0\in\mathcal{H}$, employed the sequential iterative process
$$x^{k+1}=\mathcal{T}_{C_{0},C_{1}}(x^k),\quad k\geq 0.$$
It is, thus, restricted to handling only two sets.

Borwein and Tam in \cite{bt14} introduced the cyclic-DR algorithm which is designed to
solve CFPs with more than two sets. Their cyclic-DR
algorithm applies sequentially the original\textit{ }$2$-sets-DR\textit{
}operator~\eqref{eq:DR_orig} over subsequent pairs of sets. Censor and Mansour in \cite{cm16}
extended the algorithmic structure to deal with string-averaging and
block-iterative structural regimes.

In this work we propose a cyclic DR algorithm that uses $r$-sets DR operators
and prove its convergence. We present numerical illustrations of the potential
advantage of $r$-sets DR operators with $r>2$ over the original\textit{ }%
$2$-sets-DR\textit{ }operator in this framework.
We discovered the insight how to employ $r$-sets-DR operators which hides in the cyclic DR algorithm
of Borwein and Tam \cite[Section 3]{bt14}. The Borwein-Tam cyclic DR algorithm uses $2$-sets-DR operators sequentially but for each new pair of sets it uses the last set of the previous pair as the first set in the new pair. Mimicking
this recipe enables us to use $r$-sets-DR operators in a cyclic DR algorithm.

The analysis of convergence of the algorithm presented here is quite standard and relies on tools from fixed point theory and convex analysis. So, the main contribution of the paper is the algorithmic discovery of how to properly employ $r$-sets DR operators with $r>2$ in the cyclic DR algorithm. This is a theoretical development that shows that the Borwein and Tam cyclic DR algorithm is a special case of the more general framework proposed here. This opens the door for many future research questions of extending results on the Borwein and Tam cyclic DR algorithm to the new $r$-sets DR operators with $r>2$ framework.

The paper is organized as follows. In Section \ref{sec:Prelim} we present
definitions and notions needed in the sequel. In Section \ref{sec:Algs} the
$r$-sets-DR operator and cyclic algorithm are given and the algorithm's
convergence is analyzed. Finally, in Section \ref{sec:Num}
numerical illustrations demonstrate the potential advantage of $r$-sets DR operators with $r>2$.

\section{Preliminaries\label{sec:Prelim}}

Let $\mathcal{H}$ be a real Hilbert space with inner product $\langle
\cdot,\cdot\rangle$ and norm $\Vert\cdot\Vert$, and let $D$ be a nonempty,
closed and convex subset of $\mathcal{H}$. We write $x^{k}\rightharpoonup x$
to indicate that the sequence $\left\{  x^{k}\right\}  _{k=0}^{\infty}$
converges weakly to $x$, and $x^{k}\rightarrow x$ to indicate that the sequence
$\left\{  x^{k}\right\}  _{k=0}^{\infty}$ converges strongly to $x.$ We start
by recalling the definition and properties of the metric projection operator.
For each point $x\in\mathcal{H},$\ there exists a unique nearest point in $D$,
denoted by $P_{D}(x)$. That is,%
\begin{equation}
\left\Vert x-P_{D}\left(  x\right)  \right\Vert \leq\left\Vert x-y\right\Vert
,\text{ for all }y\in D.
\end{equation}
The mapping $P_{D}:\mathcal{H}\rightarrow D,$ called \textit{the metric
projection of }$H$\textit{ onto }$D$, is well-known,
see for example \cite[Fact 1.5(i)]{bb96}, to be \textit{firmly} \textit{nonexpansive}, thus, \textit{nonexpansive}, see Definition \ref{def:FNE} below. The metric projection $P_{D}$ is characterized
\cite[Section 3]{gr84} by the facts that $P_{D}(x)\in D$ and%
\begin{equation}
\left\langle x-P_{D}\left(  x\right)  ,P_{D}\left(  x\right)  -y\right\rangle
\geq0,\text{ for all }x\in\mathcal{H},\text{ }y\in D. \label{eq:ProjP1}%
\end{equation}
If $D$ is a hyperplane, or even a closed affine subspace,
then (\ref{eq:ProjP1}) becomes an equality.

All items in the next definition can be found, e.g., in Cegielski's excellent
book \cite{Cegielski12}.

\begin{definition}
\label{def:FNE} Let $h:\mathcal{H}\rightarrow\mathcal{H}$ be an operator and
let $D\subset\mathcal{H}.$

(i) The operator $h$ is called \texttt{Lipschitz continuous} on $D\subset
\mathcal{H}$ with constant $L>0$ if%
\begin{equation}
\Vert h(x)-h(y)\Vert\leq L\Vert x-y\Vert,\text{\ for all\ }x,y\in D.
\end{equation}

(ii) The operator $h$ is called \texttt{nonexpansive} on $D$ if it is
$1$-Lipschitz continuous.\newline

(iii) The operator $h$ is called \texttt{firmly nonexpansive}
\texttt{\cite{gr84}} on $D$ if%
\begin{equation}
\left\langle h(x)-h(y),x-y\right\rangle \geq\left\Vert h(x)-h(y)\right\Vert
^{2},\text{ for all }x,y\in D\text{.}%
\end{equation}

(iv) The operator $h$ is called \texttt{averaged} \cite{bbr78} if there exists
a nonexpansive operator $N:\mathcal{H}\rightarrow\mathcal{H}\ $and a number
$c\in(0,1)$ such that%
\begin{equation}
h=(1-c)\Id+cN.
\end{equation}
In this case, we say that $h$ is $c$-av \cite{byrne04}.\newline

(v) A nonexpansive operator $h$ satisfies \texttt{Condition (W)} \cite{dr92}
if whenever $\{x^{k}-y^{k}\}_{k=1}^{\infty}$ is bounded and $\Vert x^{k}%
-y^{k}\Vert-\Vert h(x^{k})-h(y^{k})\Vert\rightarrow0$, it follows that
$(x^{k}-y^{k})-(h(x^{k})-h(y^{k}))\rightharpoonup0.$\newline

(vi) The operator $h$ is called \texttt{strongly nonexpansive} \cite{br77} if
it is nonexpansive and whenever $\{x^{k}-y^{k}\}_{k=1}^{\infty}$ is bounded
and $\Vert x^{k}-y^{k}\Vert-\Vert h(x^{k})-h(y^{k})\Vert\rightarrow0$, it
follows that $(x^{k}-y^{k})-(h(x^{k})-h(y^{k}))\rightarrow0$.\newline
\end{definition}

\begin{definition}
\label{def:QNE} Let $h:\mathcal{H}\rightarrow\mathcal{\mathcal{H}}$
be an operator with $\operatorname*{Fix}(h):=\{x\in
\mathcal{H}\mid h(x)=x\}\neq\emptyset$ and let $D\subseteq\mathcal{H}$ be a
nonempty, closed and convex set.

(i) The operator $h$ is called \texttt{quasi-nonexpansive} (QNE) if for all
$x\in\mathcal{\mathcal{H}}$ and all $z\in\operatorname*{Fix}(h)$,%
\begin{equation}
\Vert h(x)-z\Vert\leq\Vert x-z\Vert. \label{eq:qne}%
\end{equation}

(ii) A sequence $\{x_{k}\}_{k=0}^{\infty}\subset\mathcal{H}$ is said to be
\texttt{Fej\'{e}r-monotone} with respect to $D$, if for all $k\geq0$,%
\begin{equation}
\Vert x^{k+1}-u\Vert\leq\Vert x^{k}-u\Vert,\text{ for any }u\in D.
\label{eq:FM}%
\end{equation}

\end{definition}

Some of the relations between the above classes of operators are collected in
the following lemma. For more details and proofs, see\ Bruck and Reich
\cite{br77},\ Baillon et al. \cite{bbr78}, Goebel and Reich \cite{gr84}, Byrne
\cite{byrne04} and Combettes \cite{combettes04}.

\begin{lemma}
\label{Lem:av}(i) The operator $h:\mathcal{H}\rightarrow\mathcal{H}$ is firmly
nonexpansive, if and only if it is $1/2$-averaged.\newline

(ii) If $h_{1}$ and $h_{2}$ are $c_{1}$-av and $c_{2}$-av, respectively, then
their composition $S=h_{1}h_{2}$ is $(c_{1}+c_{2}-c_{1}c_{2})$-av.\newline

(iii) If $h_{1}$ and $h_{2}$ are averaged and $\operatorname*{Fix}(h_{1}%
)\cap\operatorname*{Fix}(h_{2})\neq\emptyset$, then%
\begin{equation}
\operatorname*{Fix}(h_{1})\cap\operatorname*{Fix}(h_{2})=\operatorname*{Fix}%
(h_{1}h_{2})=\operatorname*{Fix}(w_{1}h_{1}+w_{2}h_{2})
\end{equation}
with $w_{1}+w_{2}=1$, $w_{1},w_{2}\in$ (0,1). This result can
be generalized for any finite number of averaged operators, see, e.g.,
\cite[Lemma 2.2]{combettes04}.\newline

(iv) Every averaged operator is strongly nonexpansive and, therefore,
satisfies condition (W).
\end{lemma}

Another useful property of a sequence of operators is the following, see,
e.g., \cite[Definition 3.6.1]{Cegielski12}.

\begin{definition}
\label{def:asymp-reg}Let $\{U_{j}\}_{j=1}^{\infty}$ be a sequence of operators
$U_{j}:\mathcal{H}\rightarrow\mathcal{H}$ and denote $T_{\ell}=U_{\ell}%
U_{\ell-1},\dots U_{1}$. We say that $\{U_{j}\}_{j=1}^{\infty}$ is
\texttt{asymptotically regular} if
\begin{equation}
\lim_{\ell\rightarrow\infty}\Vert T_{\ell+1}(x)-T_{\ell}(x)\Vert=0,\text{\ for
all\ }x\in\mathcal{H}. \label{Eq:AsyOp}%
\end{equation}

\end{definition}

The well-known Opial Theorem \cite[Theorm 1]{Opial67}, see also \cite[Theorem
3.5.1]{Cegielski12}, is presented next.

\begin{theorem}
\label{The:Opial} Let $\mathcal{H}$ be a real Hilbert space and let
$D\subset\mathcal{H}$ be closed and convex. If $h:D\rightarrow D$ is an
averaged operator with $\operatorname*{Fix}(h)\neq\emptyset$ then, for any
$x^{0}\in D$, the sequence $\left\{  x^{k}\right\}  _{k=0}^{\infty},$
generated by $x^{k+1}=h(x^{k}),$ converges weakly to a point $x^{\ast}%
\in\operatorname*{Fix}(h)$.
\end{theorem}

\section{The $r$-sets-Douglas-Rachford operator and algorithm\label{sec:Algs}}

The $r$-sets-Douglas-Rachford ($r$-sets-DR) operator was defined in
\cite{cm16} as follows.

\begin{definition}
\label{def:m-dr}\cite[Definition 22]{cm16} Given a sequence of $r$ nonempty
closed convex sets, $r\geq2,$ $C_{0},C_{1},\ldots,C_{r-1}\subseteq\mathcal{H}%
$, define the composite reflection operator $\mathcal{V}_{C_{0},C_{1}%
,\ldots,C_{r-1}}:\mathcal{H\rightarrow\mathcal{H}}$ by%
\begin{equation}
\mathcal{V}_{C_{0},C_{1},\ldots,C_{r-1}}:=\mathcal{R}_{C_{r-1}}\mathcal{R}%
_{C_{r-2}}\cdots\mathcal{R}_{C_{0}}, \label{eq:composite ref.}%
\end{equation}
where $\mathcal{R}_{C_{i}}=2P_{C_{i}}-\Id$ is the reflection on the
corresponding $C_{i}$. The $r$-sets-DR operator $\mathcal{T}_{C_{0}%
,C_{1},\ldots,C_{r-1}}:\mathcal{H\rightarrow\mathcal{H}}$ is defined by%
\begin{equation}
\mathcal{T}_{C_{0},C_{1},\ldots,C_{r-1}}:=\frac{1}{2}\left(  \Id%
+\mathcal{V}_{C_{0},C_{1},\ldots,C_{r-1}}\right)  . \label{eq:r-sets-DR op.}%
\end{equation}

\end{definition}

For $r=2$ the $r$-sets-DR operator coincides with the
original $2$-sets-DR operator \eqref{eq:DR_orig} and when it is applied sequentially repeatedly on two sets
$m=2$ the original DR algorithm is recovered. For $r=3$ the $r$-sets-DR operator coincides with the $3$-sets-DR operator defined in \cite[Eq. (2)]{artacho2013recent}. The question whether the $3$-sets-DR operator can be applied sequentially repeatedly on three sets $m=3$ was asked there. However, it is shown, in \cite[Example 2.1]{artacho2013recent}, that such an iterative process of the form%
\begin{equation}
x^{k+1}=\mathcal{T}_{C_{0},C_{1},C_{2}}(x^{k})
\end{equation}
that uses $3$-sets-DR operators sequentially \textcolor{blue}{for} $m=3$ need
not generate a sequence that converges to a feasible point.

In this paper we discovered the insight how to employ
$r$-sets-DR operators which hides in the cyclic DR algorithm
of Borwein and Tam \cite[Section 3]{bt14}. The Borwein-Tam cyclic DR algorithm
uses $2$-sets-DR operators sequentially but for each new pair of sets it uses
the last set of the previous pair as the first set in the new pair. Mimicking
this recipe enables us to use $r$-sets-DR operators in a cyclic DR algorithm.

Given a CFP (\ref{P:CFP}) with $m$ sets indexed by
$0,1,\ldots,m-1$, and an integer $r\geq2,$ we compose, for any integer $d\geq1,$ the finite
sequence of sets
\begin{equation}
C_{m,r}(d):=C_{((r-1)d-(r-1))\text{mod }m},C_{((r-1)d-(r-2))\text{mod }%
m},\ldots,C_{((r-1)d)\text{mod }m},
\end{equation}
in which the individual sets belong to the family of sets of the given CFP. We further define the operator $S_{d}:\mathcal{H}\rightarrow\mathcal{H}$%
\begin{equation}
S_{d}:=\mathcal{T}_{C_{m,r}(d)}, \label{Oper:S}%
\end{equation}
performing an $r$-sets-DR operator on the sets of $C_{m,r}(d).$ We use it to
present our $r$-sets-Douglas-Rachford algorithm.

\begin{algorithm}
\label{alg:Gen-ALG}$\left.  {}\right.  $\textbf{The }$\boldsymbol{r}%
$\textbf{-sets-Douglas-Rachford cyclic Algorithm}\newline

\textbf{Step 0}: Select an arbitrary starting point $x^{0}\in\mathcal{H}$ and
set $k=0$.\newline

\textbf{Step 1}: Given the current iterate $x^{k}$, compute%
\begin{equation}
\label{Alg:step}x^{k+1}=S_{k+1}(x^{k}).
\end{equation}

\textbf{Step 2}:\textbf{\ }If\textbf{\ } $x^k=x^{k+1}=\dots=x^{k+\lceil m/r\rceil}$ (where $\lceil a\rceil$ stands for the smallest integer greater than or equal to $a$) then stop. Otherwise,
set $k\leftarrow(k+1)$ and return to \textbf{Step 1.}
\end{algorithm}

\begin{example}
\label{example:1}%
Assume that the CFP contains 5 sets
$C_{0},C_{1},C_{2},C_{3},C_{4},$ and choose $r=3.$ Then
\begin{equation}
C_{5,3}(1)=C_{((3-1)1-(3-1))\text{mod }5},C_{((3-1)1-(3-2))\text{mod }%
5},\ldots,C_{((3-1)1)\text{mod }5}=C_{0},C_{1},C_{2}%
\end{equation}
and $S_{1}=\mathcal{T}_{C_{5,3}(1)}=\mathcal{T}_{C_{0},C_{1},C_{2}}.$
Similarly, $S_{2}=\mathcal{T}_{C_{5,3}(2)}=\mathcal{T}_{C_{2},C_{3},C_{4}},$
$S_{3}=\mathcal{T}_{C_{5,3}(3)}=\mathcal{T}_{C_{4},C_{0},C_{1}},$
$S_{4}=\mathcal{T}_{C_{5,3}(4)}=\mathcal{T}_{C_{1},C_{2},C_{3}}$ and
$S_{5}=\mathcal{T}_{C_{5,3}(5)}=\mathcal{T}_{C_{3},C_{4},C_{0}},$ and so on
for all integers $d\geq1.$ This realizes the algorithmic structure stated above.
\end{example}

One way to handle the convergence proof of Algorithm \ref{alg:Gen-ALG} is to
base it on an appropriate generalization of Opial's theorem such as
\cite[Theorem 9.9]{cc11}, see also \cite[Section 3.5]{Cegielski12}. This
approach leads to the next theorem.

\begin{theorem}
\label{The:Gen_Opial} Let $C_{i}\subseteq\mathcal{H},$ for $i=0,\ldots,m-1,$
be nonempty, closed and convex sets with $\operatorname*{int}\left(
\cap_{0=1}^{m-1}C_{i}\right)  \neq\emptyset$. Let $\{S_{k}\}_{k=1}^{\infty}$
be the family of operators defined in (\ref{Oper:S}). Assume that
$S:\mathcal{H}\rightarrow\mathcal{H}$ is a nonexpansive operator with
$\operatorname*{Fix}(S)\neq\emptyset$ for which the following assumptions
hold:\newline(1) $\operatorname*{Fix}(S)\subseteq\left(  \cap_{k=1}^{\infty
}\operatorname*{Fix}(S_{k})\right)  \cap\left(  \cap_{0=1}^{m-1}C_{i}\right)
,$ \newline(2) $\left\{  x^{k}\right\}  _{k=0}^{\infty},$ generated by
Algorithm \ref{alg:Gen-ALG}, is \textit{Fej\'{e}r-monotone} with respect to
$\operatorname*{Fix}(S)$,\newline(3) the inequality $\Vert S_{k}(x^{k}%
)-x^{k}\Vert\geq\beta\Vert S(x^{k})-x^{k}\Vert$ is satisfied for all $k\geq0,$
for some $\beta>0$.\newline Then the sequence $\left\{  x^{k}\right\}
_{k=0}^{\infty},$ generated by Algorithm \ref{alg:Gen-ALG}, converges weakly
to a point $x^{\ast}\in\operatorname*{Fix}(S)$, and, in particular, $x^{\ast
}\in\cap_{0=1}^{m-1}C_{i}$.
\end{theorem}

\begin{proof}
We first show that the family of operators $\{S_{k}\}_{k=1}^{\infty}$, defined
in (\ref{Oper:S}), is quasi-nonexpansive and asymptotically regular
(Definitions \ref{def:QNE} and \ref{def:asymp-reg} above). Let $d\in\mathbb{\mathbb{N}}$,
then the composition of reflections operator $\mathcal{V}_{C_{m,r}(d)}$ is
nonexpansive and hence $\mathcal{T}_{C_{m,r}(d)}$ is firmly-nonexpansive
($1/2$-averaged). Thus, this operator is also asymptotically regular, see,
e.g., the discussion following Theorem 9.7 in \cite{cc11}. The asymptotic
regularity of $\{S_{k}\}_{k=1}^{\infty}$ and the assumptions of the theorem
enable the use of \cite[Theorem 1]{cegielski07} (see also \cite[Theorem
9.9]{cc11} and \cite[Subsection 3.6]{Cegielski12}) to obtain the desired
result.\bigskip
\end{proof}

Since the conditions of Theorem \ref{The:Gen_Opial} are not easy to verify in
practice, we present an alternative convergence result for Algorithm
\ref{alg:Gen-ALG}. Given $m$ nonempty, closed and convex sets $C_{i}$, for
$i=0,1,\ldots,m-1$ and $1<r\leq m-1$, we look at the string of $(r-1)m$ sets
that is composed of $r-1$ copies of $\{C_{0},C_{1},\ldots,C_{m-1}\},$ i.e.,%
\begin{equation}
\underbrace{C_{0},C_{1},\ldots,C_{m-1}}_{1},\underbrace{C_{0},C_{1}%
,\ldots,C_{m-1}}_{2},\ldots,\underbrace{C_{0},C_{1},\ldots,C_{m-1}}_{r-1},
\label{eq:3.7}%
\end{equation}
and define with (\ref{Oper:S}) the composite operator $Q$:%
\begin{equation}
Q:=S_{m}\cdots S_{2}S_{1}. \label{Eq:Op_S1}%
\end{equation}

\begin{example}
To continue Example \ref{example:1}, here (\ref{eq:3.7})
takes the form:%
\begin{equation}
\underbrace{C_{0},C_{1},\ldots,C_{4}}_{1},\underbrace{C_{0},C_{1},\ldots
,C_{4}}_{2}%
\end{equation}
and the operator $Q$ of (\ref{Eq:Op_S1}) is:%
\begin{equation}
Q:=S_{5}S_{4}S_{3}S_{2}S_{1}=\mathcal{T}_{C_{3},C_{4},C_{0}}\mathcal{T}%
_{C_{1},C_{2},C_{3}}\mathcal{T}_{C_{4},C_{0},C_{1}}\mathcal{T}_{C_{2}%
,C_{3},C_{4}}\mathcal{T}_{C_{0},C_{1},C_{2}}.
\end{equation}
This kind of algorithmic operator guarantees that the last set that is handled is $C_{0}$.
\end{example}

We will prove the convergence of Algorithm \ref{alg:Gen-ALG} with $S_{k}$ in
(\ref{Alg:step}) replaced by $Q,$ for all $k\geq1.$ We need the following
lemma which is based on \cite[Corollary 23]{cm16}.

\begin{lemma}
Let $C_{i}\subseteq\mathcal{H},$ for $i=0,1,\ldots,m-1,$ be nonempty, closed
and convex sets with $\operatorname*{int}\left(  \cap_{i=0}^{m-1}C_{i}\right)
\neq\emptyset$. For fixed $r\in\{2,3,\ldots,m-1\}$, we have%
\begin{equation}
\cap_{i=0}^{r-1}C_{i}=\operatorname*{Fix}(\mathcal{T}_{C_{0},C_{1}%
,\ldots,C_{r-1}}). \label{eq:20}%
\end{equation}

\end{lemma}

\begin{proof}
Obviously,%
\begin{equation}
\emptyset\neq\operatorname*{int}\left(  \cap_{i=0}^{m-1}C_{i}\right)
\subseteq\operatorname*{int}\left(  \cap_{i=0}^{r-1}C_{i}\right)  .
\end{equation}
Since%
\begin{equation}
\operatorname*{Fix}(\mathcal{T}_{C_{0},C_{1},\ldots,C_{r-1}}%
)=\operatorname*{Fix}(\mathcal{V}_{C_{0},C_{1},\cdots,C_{r-1}})=\cap
_{i=0}^{r-1}\operatorname*{Fix}(\mathcal{R}_{C_{i}})=\cap_{i=0}^{r-1}C_{i},
\label{eq:Fix_r}%
\end{equation}
combining the above, we get (\ref{eq:20}) as desired.\bigskip
\end{proof}

The alternative convergence result of Algorithm \ref{alg:Gen-ALG} follows.

\begin{theorem}
\label{The:Main} Let $C_{i}\subseteq\mathcal{H},$ for $i=0,\ldots,m-1,$ be
nonempty, closed and convex sets. If $\operatorname*{int}\left(  \cap
_{0=1}^{m-1}C_{i}\right)  \neq\emptyset$ then any sequence $\left\{
x^{k}\right\}  _{k=0}^{\infty},$ generated by Algorithm \ref{alg:Gen-ALG} with
$S_{k}$ replaced by $Q$ as in (\ref{Eq:Op_S1}), converges weakly to a point
$x^{\ast}$ which solves the convex feasibility problem (\ref{P:CFP}).
\end{theorem}

\begin{proof}
Let $r\in\{2,3,\ldots,m-1\}$. Since the operator $\mathcal{V}_{C_{1}%
,C_{2},\cdots,C_{r}}$ is nonexpansive, $\mathcal{T}_{C_{1},C_{2},\ldots,C_{r}%
}$ is firmly-nonexpansive, i.e., $1/2$-averaged. Since composition of averaged
operators is averaged, we get that any operator $S_{d}$ (\ref{Oper:S}) is
averaged and so is also $Q$ of (\ref{Eq:Op_S1}).

Next, we study $\operatorname*{Fix}(Q)$. Since $\cap_{i=0}^{m-1}C_{i}%
\neq\emptyset$, Lemma \ref{Lem:av}(iii) and (\ref{eq:20}) yield%
\begin{equation}
\operatorname*{Fix}(Q)=\operatorname*{Fix}(S_{m}\cdots S_{2}S_{1})=\cap
_{d=1}^{m}\operatorname*{Fix}S_{d}=\cap_{i=0}^{m-1}C_{i}. \label{eq:3.12}%
\end{equation}
The rest of the proof follows directly from the Opial theorem (Theorem
\ref{The:Opial} above) and the proof is complete.
\end{proof}

\begin{remark}
\begin{description}
  \item[(i)] In the finite-dimensional case, Theorem \ref{The:Main} implies also convergence of Algorithm \ref{alg:Gen-ALG} with $\{S_k\}_{k=1}^\infty$. 
  \item[(ii)] In the special case when $\frac{{\displaystyle}m}{{\displaystyle}r-1}=n\in\mathbb{N}$ one can define the operator $\widetilde{Q}:=S_{n}\cdots S_{2}S_{1},$ which means that $\widetilde{Q}$ preforms one full \textquotedblleft sweep\textquotedblright\ over the sets $C_{m-1},\dots,C_{0},$ and use it instead of $Q.$
\end{description}
\end{remark}

\begin{definition}
Let $\mathbb{N}$ be the set of natural numbers, $\{h_{1},h_{2},\ldots\}$ be a
sequence of operators, and $r:\mathbb{N}\rightarrow\mathbb{N}$. An
unrestricted (or random) product of these operators is the sequence
$\{S_{n}\}_{n\in\mathbb{N}}$ defined by $S_{n}:=h_{r(n)}h_{r(n-1)}\cdots
h_{r(1)}$.
\end{definition}

We recall the following result by Dye and Reich.

\begin{theorem}
\label{Th:DR}\cite[Theorem 1]{dr92} Let $T_{1}:\mathcal{H}\rightarrow
\mathcal{H}$ and $T_{2}:\mathcal{H}\rightarrow\mathcal{H}$ be two (W)
nonexpansive mappings on a Hilbert space $\mathcal{H}$, whose fixed point sets
have a nonempty intersection. Then any random product $\{S_{n}\}_{n\in
\mathbb{N}}$, from $T_{1}$ and $T_{2}$ converges weakly (to a common fixed point).
\end{theorem}

With the aid of this theorem we can prove that products of projection
operators may be interlaced between the $r$-sets-DR operators in Algorithm
\ref{alg:Gen-ALG}.

\begin{theorem}
\label{The:Main2} Let $C_{i}\subseteq\mathcal{H},$ for $i=0,1,\ldots,m-1,$ be
nonempty, closed and convex sets with $\operatorname*{int}\left(  \cap
_{0=1}^{m-1}C_{i}\right)  \neq\emptyset$. Given the operators $T_{1}=Q$ (where
$Q$ is defined in (\ref{Eq:Op_S1})) and $T_{2}=P_{C_{0}}P_{C_{1}}\cdots
P_{C_{m-1}}$, any sequence $\left\{  x^{k}\right\}  _{k=0}^{\infty}$,
generated by any random product from $T_{1}$ and $T_{2},$ converges weakly to
a point $x^{\ast}$ which solves the CFP (\ref{P:CFP}).
\end{theorem}

\begin{proof}
By (\ref{eq:3.12})%
\begin{equation}
\operatorname*{Fix}(T_{1})=\operatorname*{Fix}(Q)=\cap_{i=0}^{m-1}C_{i}%
\end{equation}
and clearly also%
\begin{equation}
\operatorname*{Fix}(T_{2})=\operatorname*{Fix}(P_{C_{0}}P_{C_{1}}\cdots
P_{C_{m-1}})=\cap_{i=0}^{m-1}C_{i},
\end{equation}
yielding $\operatorname*{Fix}(T_{1})\cap\operatorname*{Fix}(T_{2})=\cap
_{i=0}^{m-1}C_{i}\neq\emptyset$. Since (see the proof of Theorem
\ref{The:Main}) the operator $Q$ is $1/2$-averaged we use Lemma \ref{Lem:av}%
(iv), to know that it satisfies condition (W). Since $T_{2}$ is also averaged,
it also satisfies condition (W). Applying Theorem \ref{Th:DR} the desired
result is obtained.
\end{proof}

\begin{remark}
Theorem \ref{The:Main2} is established with $T_{2}=P_{C_{0}}P_{C_{1}}\cdots
P_{C_{m-1}}$, but as a matter of fact, any (W) nonexpansive operator can be
chosen as long as $\operatorname*{Fix}(T_{2})=\cap_{i=0}^{m-1}C_{i}$, for
example $\mathcal{T}_{C_{0},C_{1},\cdots,C_{m-1}}$ ((\ref{eq:r-sets-DR op.})
with $r=m$).
\end{remark}

\begin{remark}
In \cite{ac17} a generalized DR operator, called the \textit{averaged
alternating modified reflections }(AAMR) operator, is introduced. It allows to
choose any parameters $\alpha\in(0,1)$ and $\beta\in(0,1)$ in the operator
$\mathcal{T}_{A,B,\alpha,\beta}:\mathcal{H}\rightarrow\mathcal{H}$ given by%
\begin{equation}
\mathcal{T}_{A,B,\alpha,\beta}:=(1-\alpha)\Id+\alpha(2\beta P_{B}-\Id)(2\beta
P_{A}-\Id)
\end{equation}
where $A$ and $B$ are nonempty, closed and convex sets. We conjecture that our
analysis given here can be properly expanded to include $r$-sets-AAMR
operators but we leave it for future work.
In this respect, it is worthwhile to note that the condition $\operatorname*{int}\left(  \cap_{0=1}^{m-1}C_{i}\right)  \neq\emptyset$ seems to be too restrictive. Probably additional convergence properties can be derived by relaxing it. For instance, in finite-dimensions, under a less restrictive condition, linear convergence results are proved in \cite{dao2018} for the cyclic 2-sets-DR algorithm with a generalized DR operator.
\end{remark}

\section{Numerical demonstrations\label{sec:Num}}

We set out to investigate and verify whether $r$-sets-DR operators with $r>2$
in a cyclic DR algorithm applied to a CFP are advantageous in any way over
the cyclic DR algorithm with $r=2$ proposed in \cite[Section 3]{bt14}. %
Our numerical illustrations demonstrate the potential advantage of $r$-sets DR operators with $r>2$, especially when the number of sets is large.

Additionally, we included in our numerical experiments also the ``Product Space Douglas--Rachford'' algorithm, which is based on Pierra's product space formulation~\cite{Pierra}. The original $2$-sets DR algorithm is applied sequentially to the product set
\begin{equation}
\mathbf{C}:=\prod_{i=0}^{m-1} C_i
\end{equation}
and to the diagonal set
\begin{equation}
\mathbf{D}:=\{(x,x,\ldots,x)\in\mathcal{H}^m\mid x\in\mathcal{H}\}.
\end{equation}
The iterative process obtained in this way has the form
\begin{equation}\label{eq:product-DR}
x^{k+1}=\mathcal{T}_{\mathbf{C},\mathbf{D}}(x^k),
\end{equation}
where $\mathcal{T}_{\mathbf{C},\mathbf{D}}$ is the $2$-sets-DR operator as in~\eqref{eq:DR_orig}, see, e.g.,~\cite[Section~3]{artacho2013recent}.

We consider two types of CFPs, with linear and quadratic constraints. For each of these type of problems and each problem size, $10$ random problems were generated and solved independently.  Algorithm~\ref{alg:Gen-ALG} was run until the
stopping criterion
\begin{equation}\label{eq:stop_criterion}
\frac{\left\Vert x^{k+j}-x^{k+j-1}\right\Vert}{\left\Vert x^{k+j-1}\right\Vert}\leq10^{-12},\quad\text {for all }j=1,2,\ldots,\lceil m/r\rceil,
\end{equation}
was met. All the experiments were run in $\mathbb{R}^{n}$ (the $n$-dimensional Euclidean space) with $n=1000.$ Initialization vectors $x^{0}$ were generated by randomly uniformly picking their coordinates from the range $[-10,10].$
All codes were written in Python~2.7 and the tests were run on an Intel Core i7-4770 CPU \@3.40GHz with 32GB RAM, under Windows 10 (64-bit).

\begin{example}[Linear CFPs]\label{Ex:linear}
In this example we consider solving a system of linear equations $Ax=\textcolor{blue}{0_m}$, where $A\in\mathbb{R}^{m\times n}$, $x\in\mathbb{R}^{n}$ and $0_m\in\mathbb{R}^{m}$. Since in real-life, experiments and measurements often come with ``noise'', we investigate the performances of our algorithm for solving the perturbed system of linear inequalities $-b_i\leq \langle a^i,x\rangle\leq b_i$, $i=1,2,\ldots,m$. The coordinates of $a^i$ were randomly uniformly generated in $[-1,1]$ and then the vectors were normalized, and $b_i$ was randomly uniformly chosen in~$[0,0.1]$.
\end{example}

\begin{example}[Quadratic CFPs]\label{Ex:quadratic}  In this example we
followed the experimental setup in \cite[Section 5]{bt14} and generated CFPs
consisting of balls of various sizes.
Each ball was created by picking a ball center $a^{i}$ with coordinates randomly uniformly
generated in the range $[-5,5].$ Then a radius $b_{i}:=\left\Vert
a^{i}\right\Vert +\alpha_{i}$ was defined by adding to the center's distance
from the origin $\left\Vert a^{i}\right\Vert $ a random number uniformly
picked from the range $[0,0.1]$ guaranteeing that the ball includes the
origin, thus, yielding a consistent CFP.
\end{example}

In our first experiment we compare the product space DR algorithm~\eqref{eq:product-DR} with our cyclic $r$-sets-DR Algorithm~\ref{alg:Gen-ALG} with different values of $r$. In Figure~\ref{fig:product} we show the running times of the different methods when the number of sets of the CFP varied from 50 to 1000. The stopping criterion~\eqref{eq:stop_criterion} was also used for the product space DR algorithm, but this time only for $j=1$. Note that a logarithmic scale was employed for the y-axis. We observe that for 1000~constraints, a number which is relatively small, the product space DR algorithm was nearly 100 times slower than each of the $r$-sets-DR methods. It is not difficult to understand the main reason why this happens: it requires to work in the product space $\mathbb{R}^{m\times n}$ instead of the original space $\mathbb{R}^n$.
\begin{figure}[htp]
\begin{center}
\begin{subfigure}{\textwidth}\centering
    \includegraphics[width=0.75\textwidth]{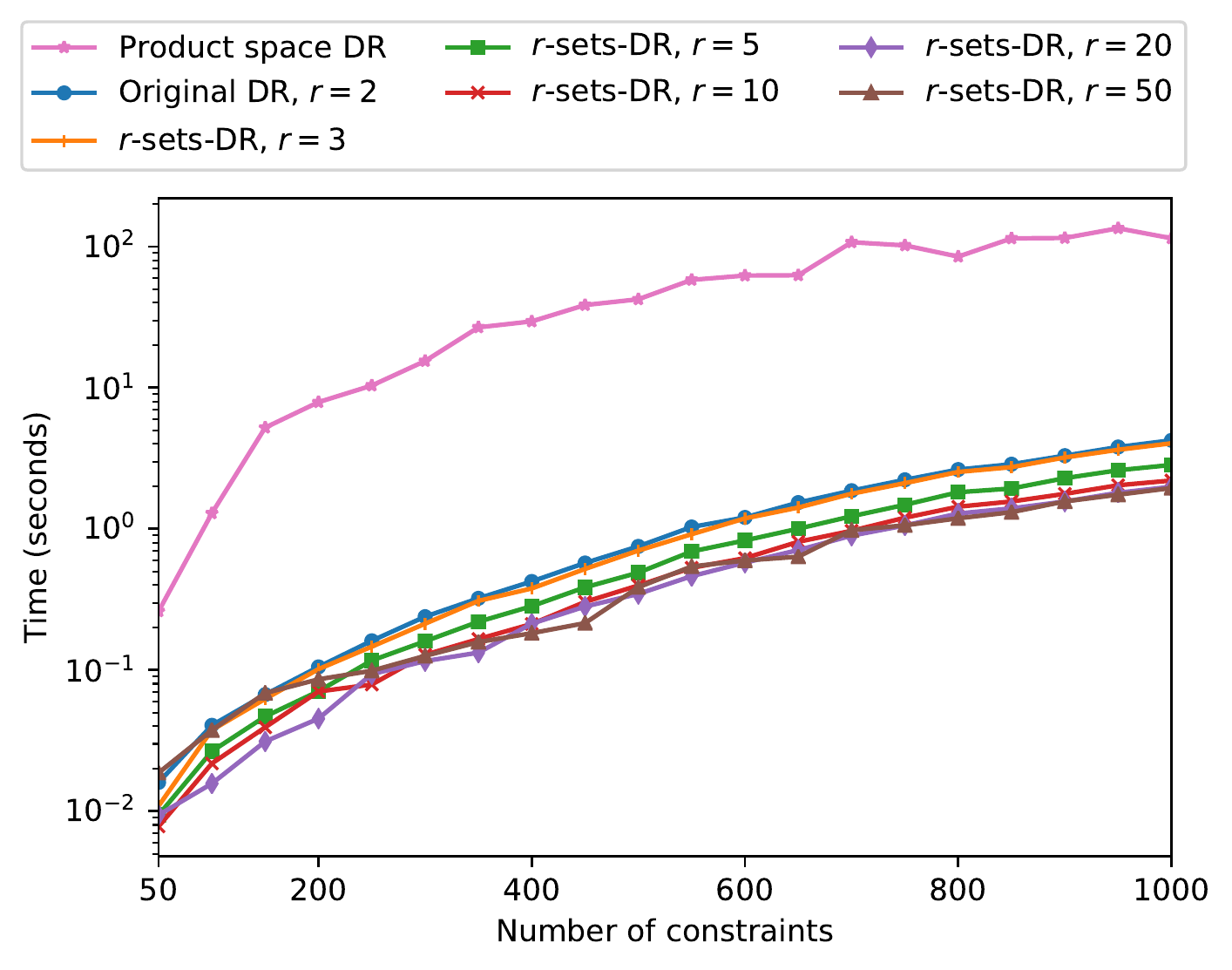}
    \caption{Linear CFPs}
\end{subfigure}
\begin{subfigure}{\textwidth}\centering
    \includegraphics[width=0.75\textwidth]{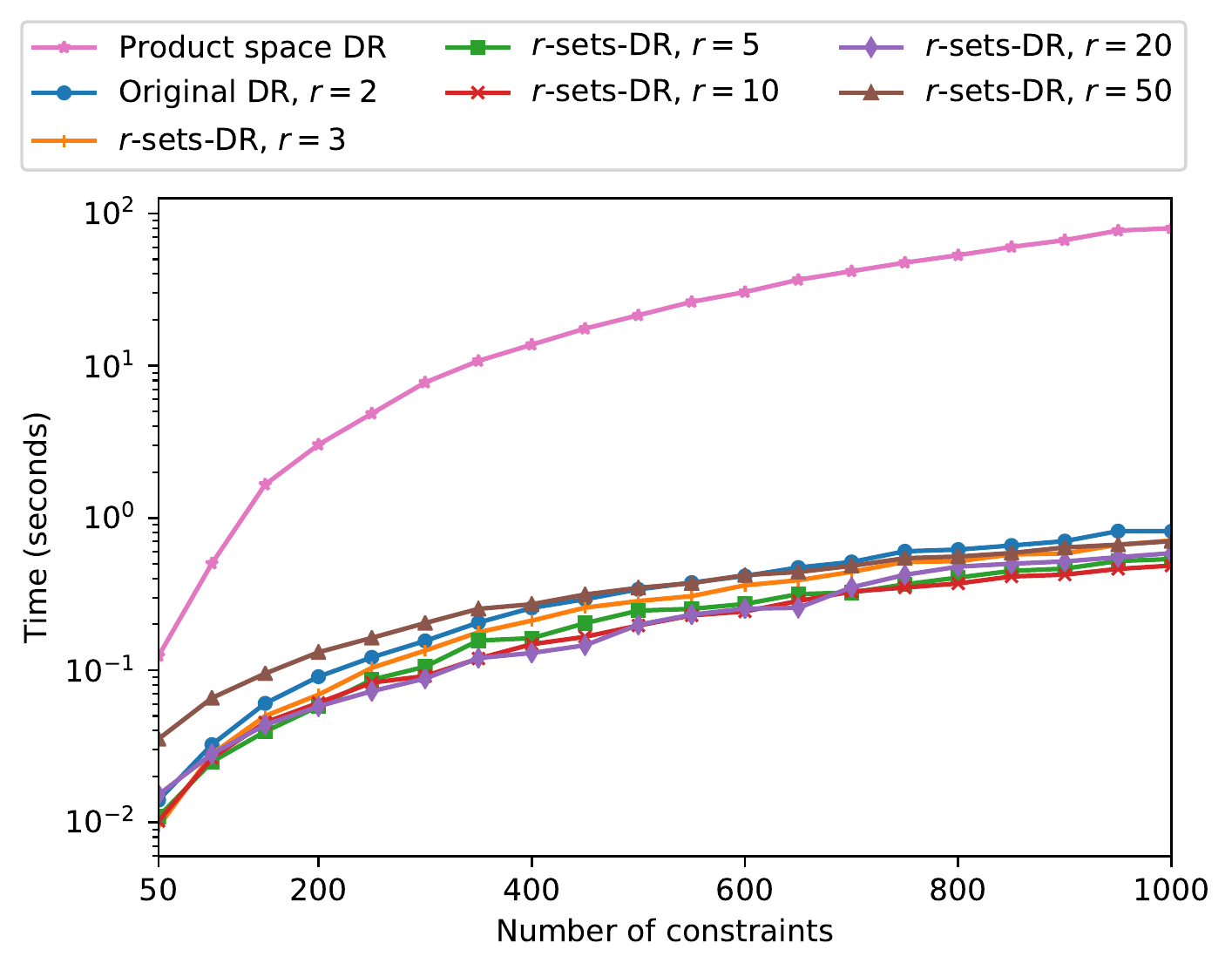}
    \caption{Quadratic CFPs}
\end{subfigure}
\end{center}
\caption{Runtimes in seconds averaged over 10 independent problems for
varying number of constraints. The product space DR algorithm is outperformed by the cyclic $r$-sets-DR algorithms.}\label{fig:product}
\end{figure}

In our second experiment, we compare our cyclic $r$-sets-DR methods for a wide range of constraints between 200 and 50,000. For each problem size, 10 independent random problems were tested. The averaged run-times are shown in Figure~\ref{fig:r-sets}. The performance profiles comparing the methods, shown in Figure~\ref{fig:pp}, were obtained as follows, see~\cite{DM02} and~\cite{bhl17}. Let $\mathcal{S}$ denote the set of all 6 solvers compared (namely, the original $2$-sets-DR scheme, and the cyclic $r$-sets-DR algorithm with $r=3,5,10,20,50$). Let $\mathcal{P}:=\{200,2500,5,000,\ldots,50,000\}$ be the set of problems. Let $t_{p,s}$ be the averaged time required to solve problem $p\in\mathcal{P}$ over the 10 random instances tested, by solver $s$.

For each problem $p$ and solver $s$, the \emph{performance ratio} is defined by
\begin{equation}
r_{p,s}:=\frac{t_{p,s}}{\min\{t_{p,s}\mid s\in \mathcal{S}\}}.
 \end{equation}
The \emph{performance profile} of a solver $s$ is a real-valued function $\pi_s:[1,+\infty)\to{[0,1]}$ defined by
\begin{equation}
\pi_s(\tau):=\frac{1}{|\mathcal{P}|}\left|\left\{p\in\mathcal{P}\mid r_{p,s}\leq \tau\right\}\right|,
\end{equation}
where $|\mathcal{P}|$ is the cardinality of the test set $\mathcal{P}$. This function indicates the probability that a performance ratio $r_{p,s}$ is within a factor $\tau$ of the best possible ratio. Thus, $\pi_s(1)$ represents the portion of problems for which solver $s\in\mathcal{S}$ has the best performance among all other solvers.

From Figures~\ref{fig:r-sets} and~\ref{fig:pp} we deduce that the cyclic DR algorithm with $r=2$ is clearly outperformed by the cyclic DR algorithms with the other $r=3,5,10,20,50$ $r$-sets-DR
operators. This trend seems even to grow and become more pronounced as the
problem sizes grow. The best performance for both linear and quadratic problems that were tested was achieved for $r=20$, closely followed by $r=10$. On average, these two algorithms were two times faster than the original cyclic DR algorithm.
\begin{figure}[htp]
\begin{center}
\begin{subfigure}{\textwidth}\centering
    \includegraphics[width=0.8\textwidth]{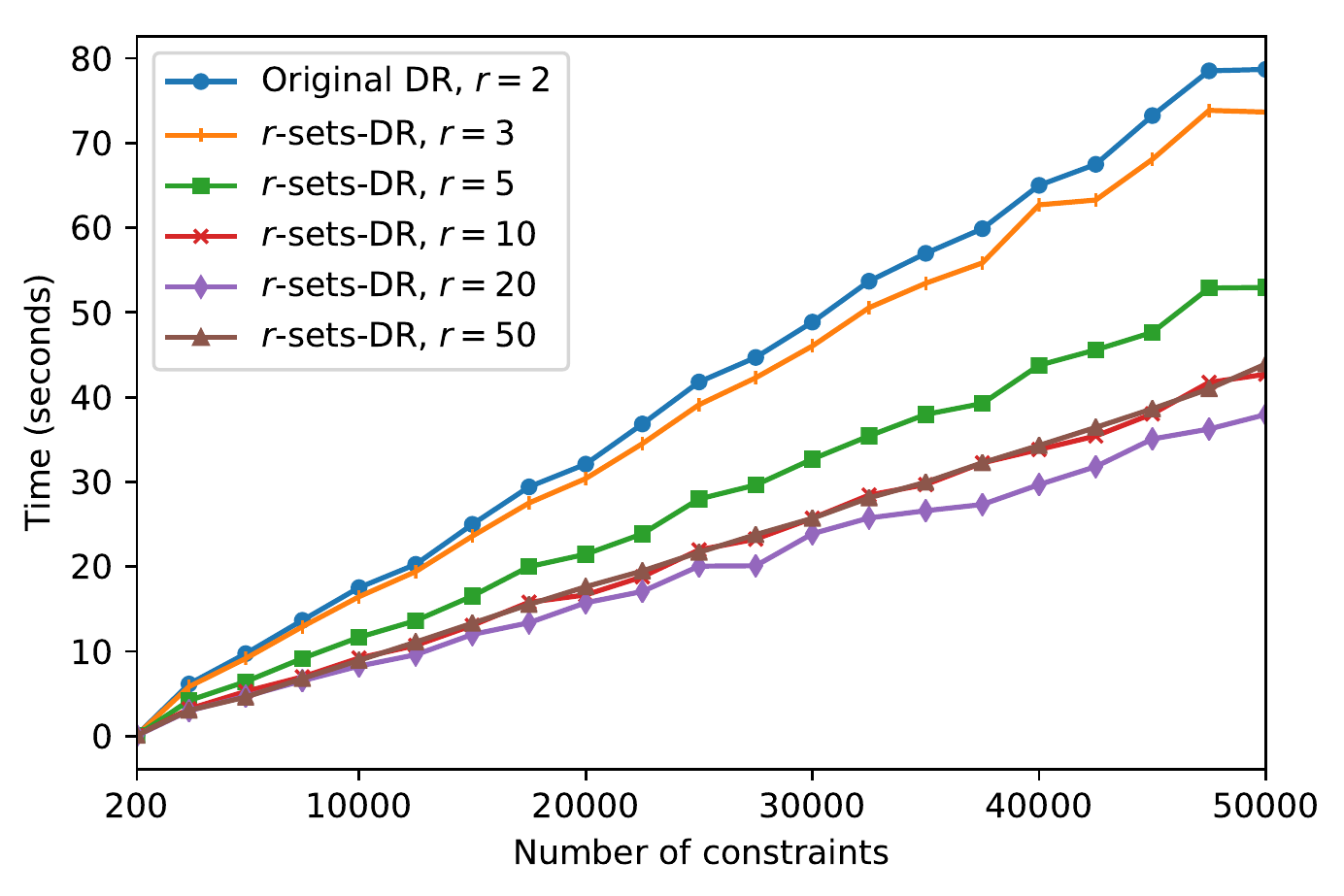}
    \caption{Linear CFPs}
\end{subfigure}
\begin{subfigure}{\textwidth}\centering
    \includegraphics[width=0.8\textwidth]{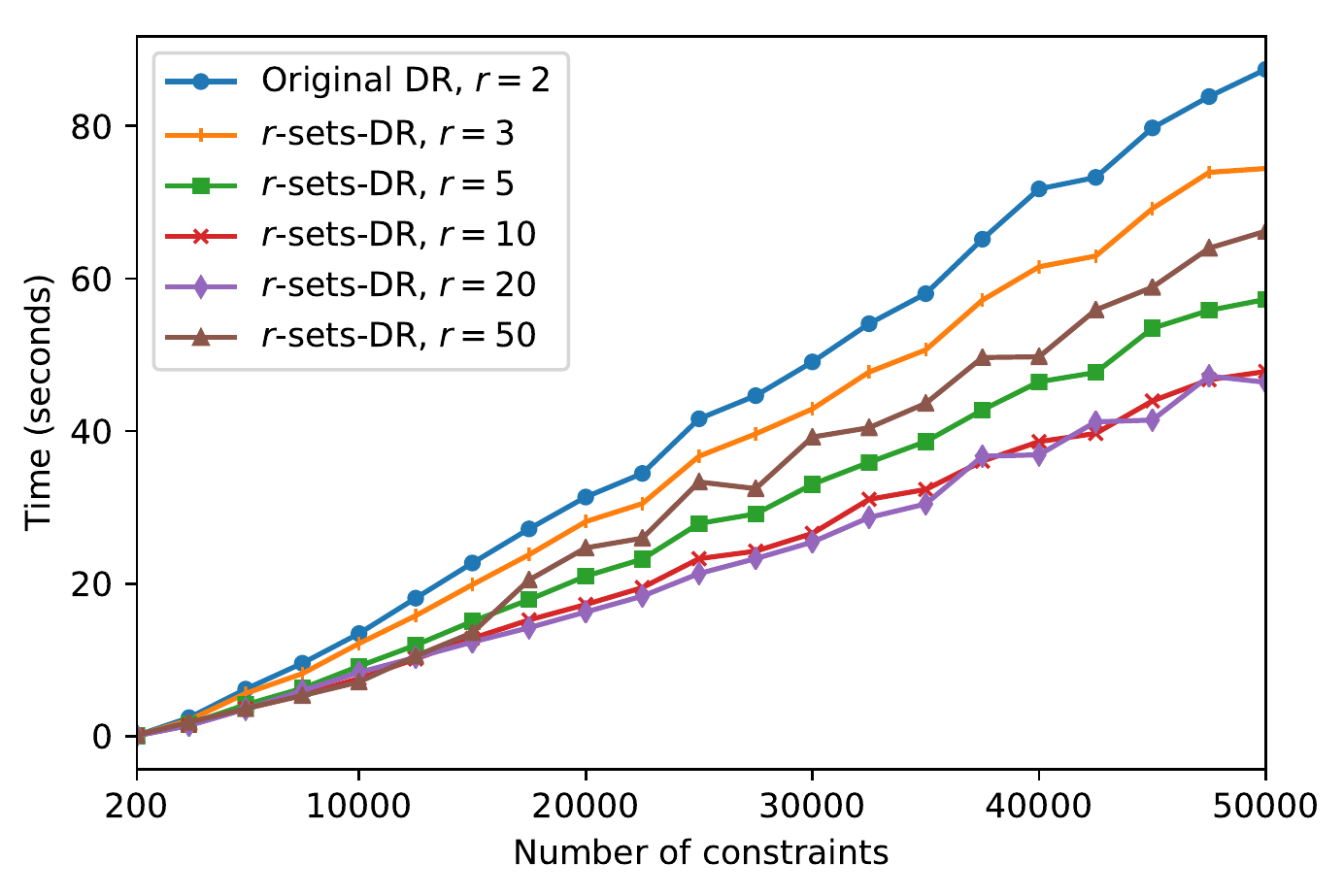}
    \caption{Quadratic CFPs}
\end{subfigure}
\caption{Runtimes in seconds averaged over 10 independent problems for
varying number of constraints. The cyclic DR algorithm with $r = 2$ is outperformed by the cyclic DR algorithms with the other $r = 3$, $5$, $10$, $20$, $50$ $r$-sets-DR operators.}
\label{fig:r-sets}
\end{center}
\end{figure}

\begin{figure}[htp]
\begin{center}
\begin{subfigure}{\textwidth}\centering
    \includegraphics[width=0.7\textwidth]{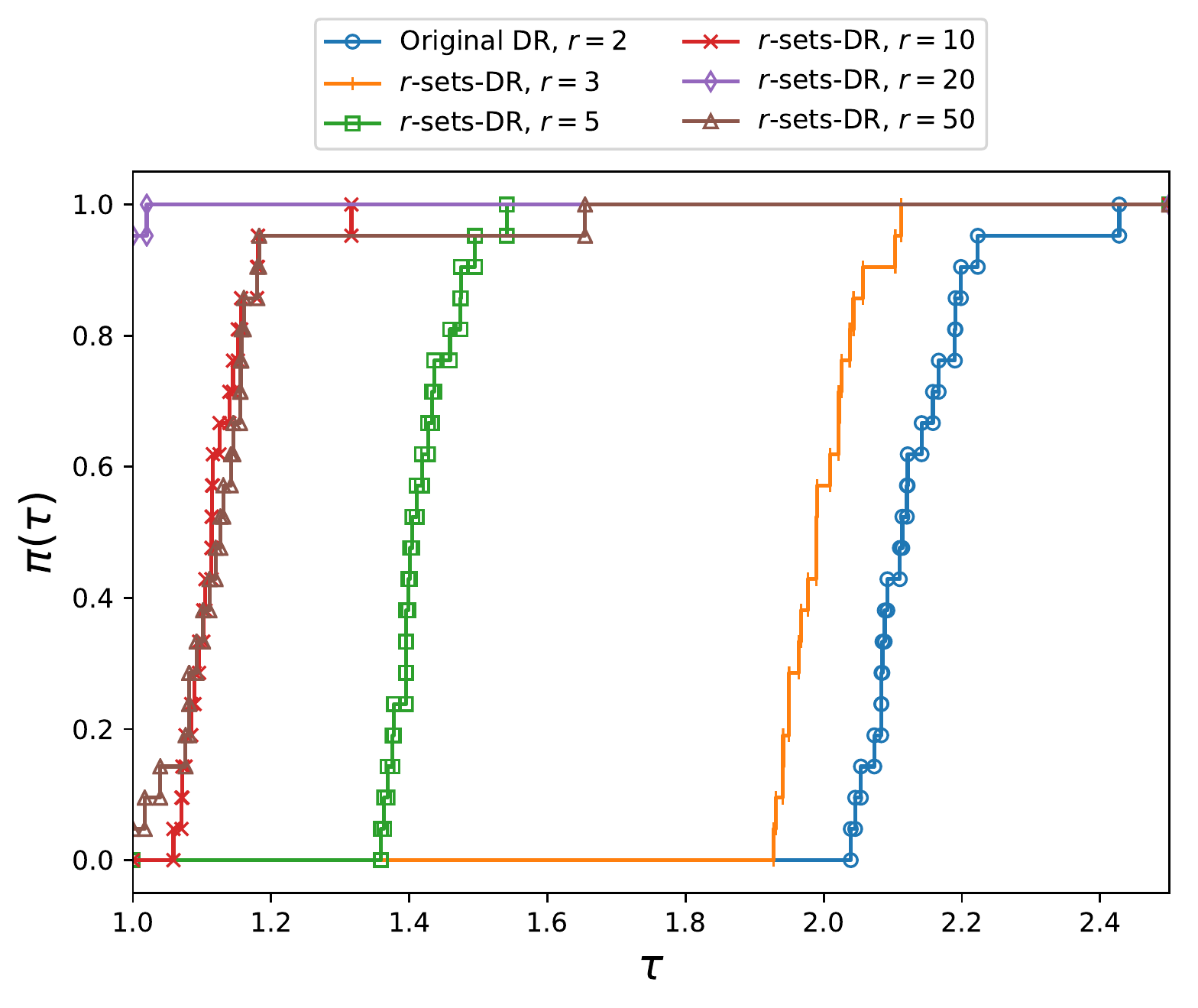}
    \caption{Linear CFPs}
\end{subfigure}
\begin{subfigure}{\textwidth}\centering
    \includegraphics[width=0.7\textwidth]{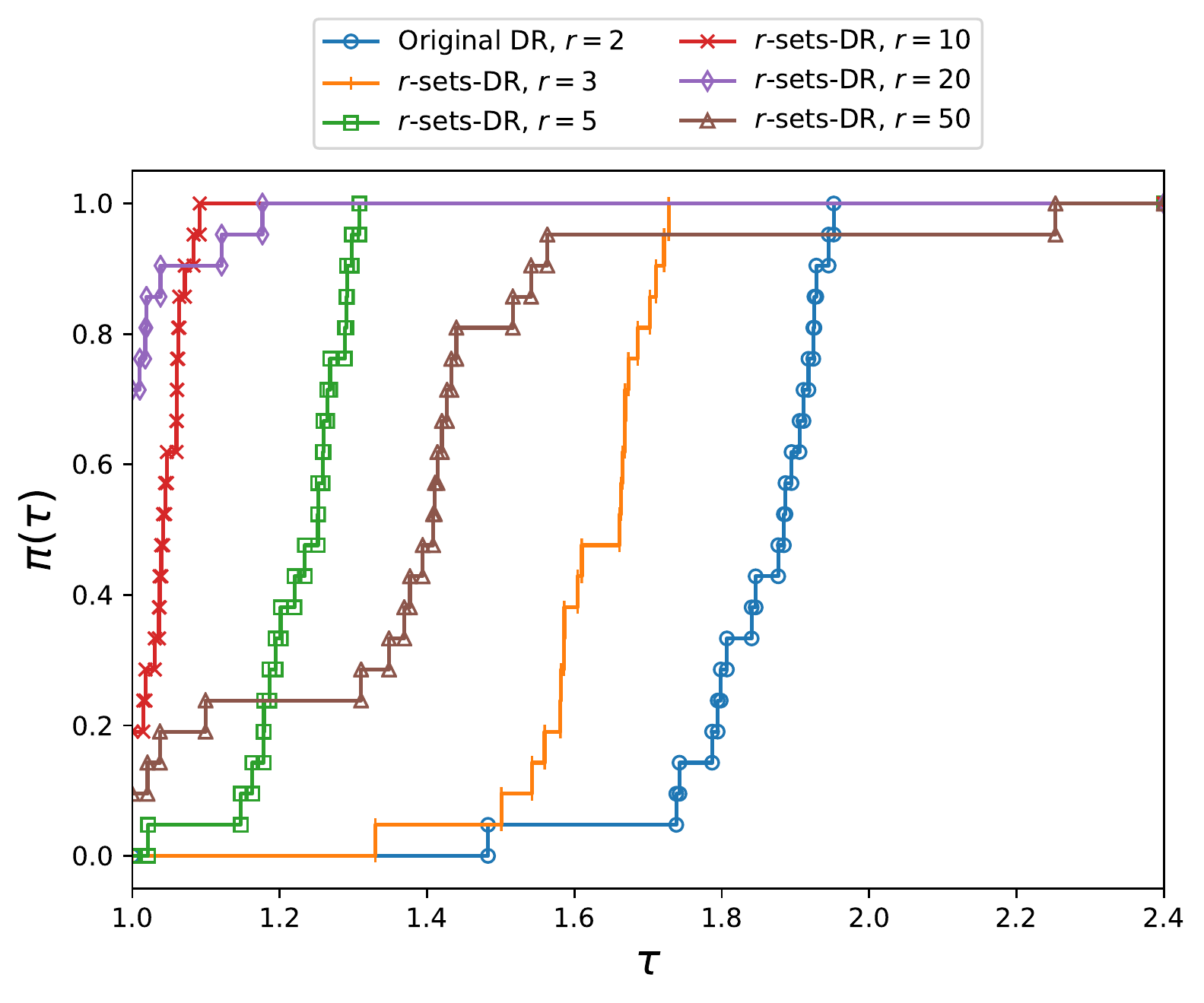}
    \caption{Quadratic CFPs}
\end{subfigure}
\caption{Performance profiles over 10 independent problems for varying number of constraints. The cyclic DR algorithm with $r = 2$ is outperformed by the cyclic DR algorithms with the other $r = 3$, $5$, $10$, $20$, $50$ $r$-sets-DR
operators.}
\label{fig:pp}
\end{center}
\end{figure}

In our last experiment, we compare the values of
\begin{equation}\label{eq:def_error}
{\rm Error}(x^k):=\sum_{i=0}^{m-1}\|P_{C_i}(x^k)-x^k\|
\end{equation}
with respect to the number of iterations and projections employed by each of the methods in one particular random experiment with 10,000 constraints. Of course, the larger $r$ is, the more projections the method uses to compute each iteration. The results, which are presented in Figure~\ref{fig:dist}, clearly show that the original cyclic DR scheme with $r=2$ uses two times more projections than the $r$-sets-DR method with $r=10,20$ or $50$ to achieve the same accuracy.
\begin{figure}[htp]
\begin{center}
\includegraphics[width=.85\textwidth]{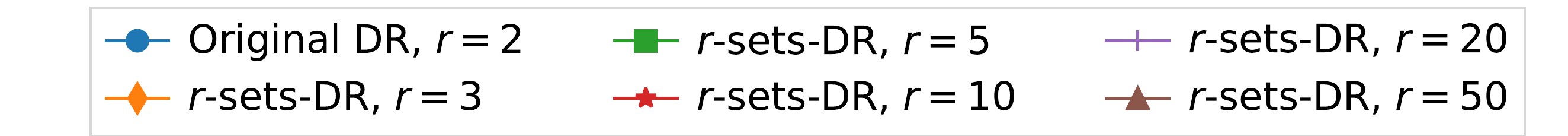}
\begin{subfigure}{0.49\textwidth}
    \includegraphics[width=\textwidth]{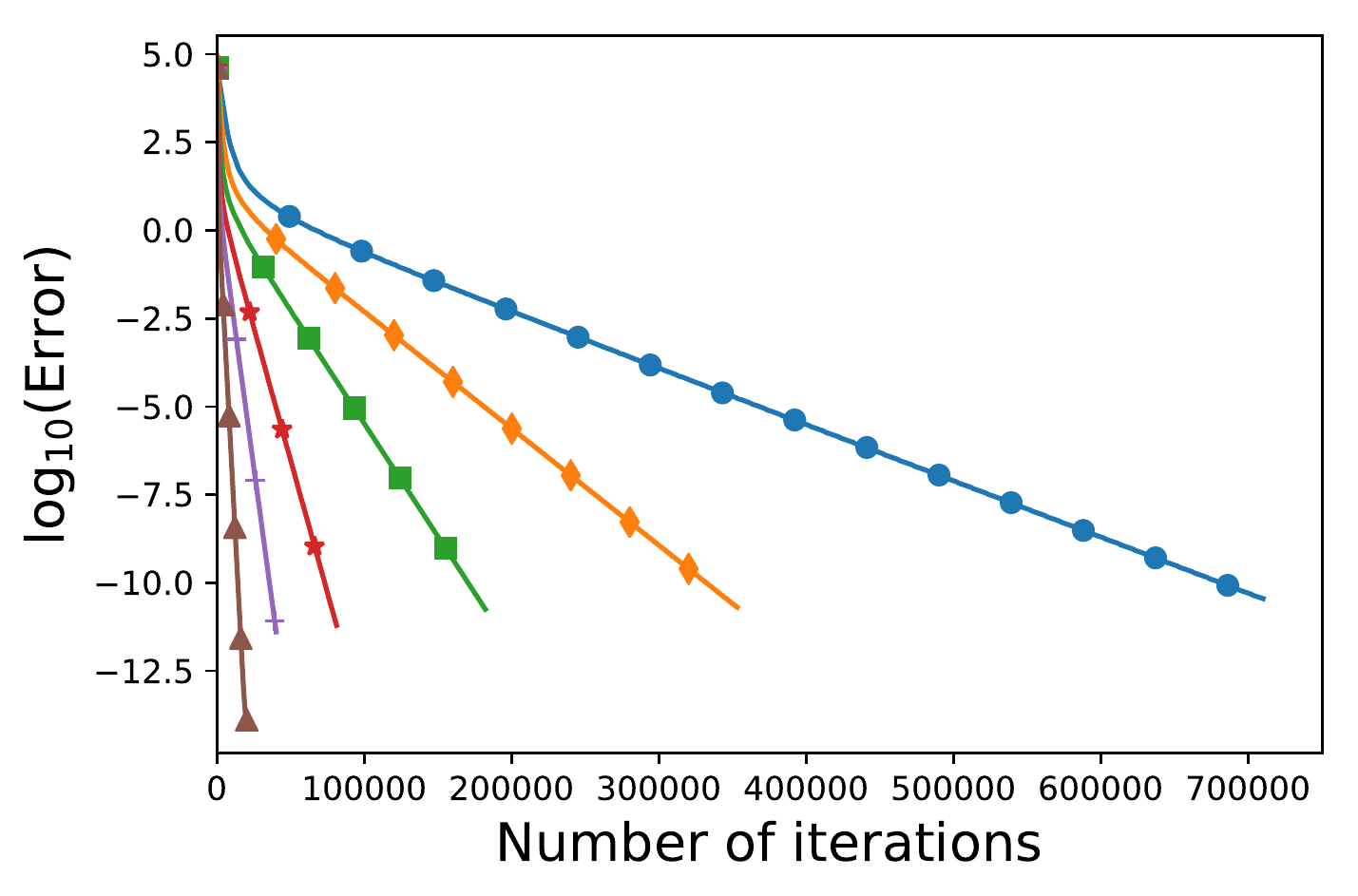}
    \caption{Linear CFPs (iterations)}
\end{subfigure}
\begin{subfigure}{0.49\textwidth}
    \includegraphics[width=\textwidth]{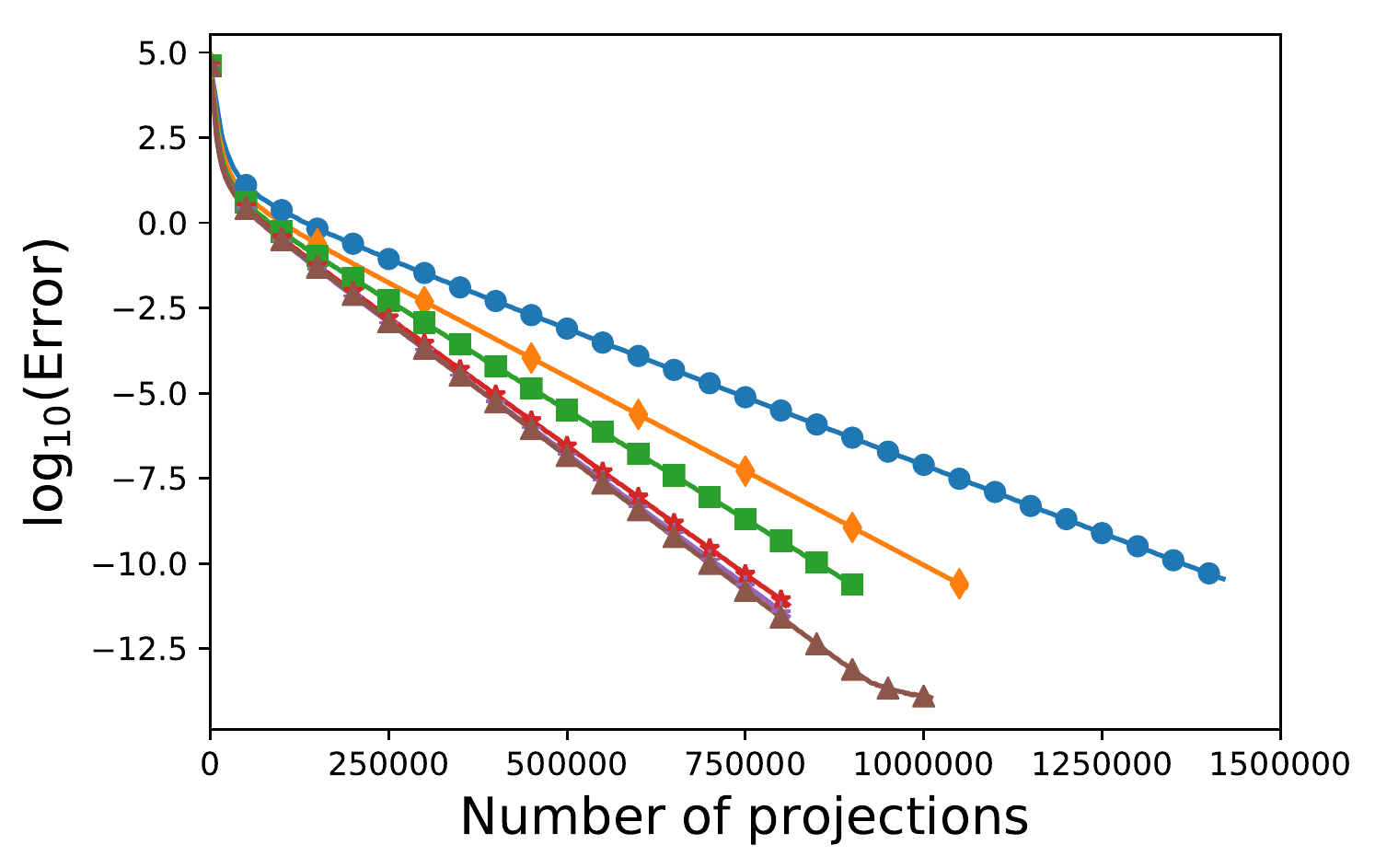}
    \caption{Linear CFPs (projections)}
\end{subfigure}
\begin{subfigure}{0.49\textwidth}
    \includegraphics[width=\textwidth]{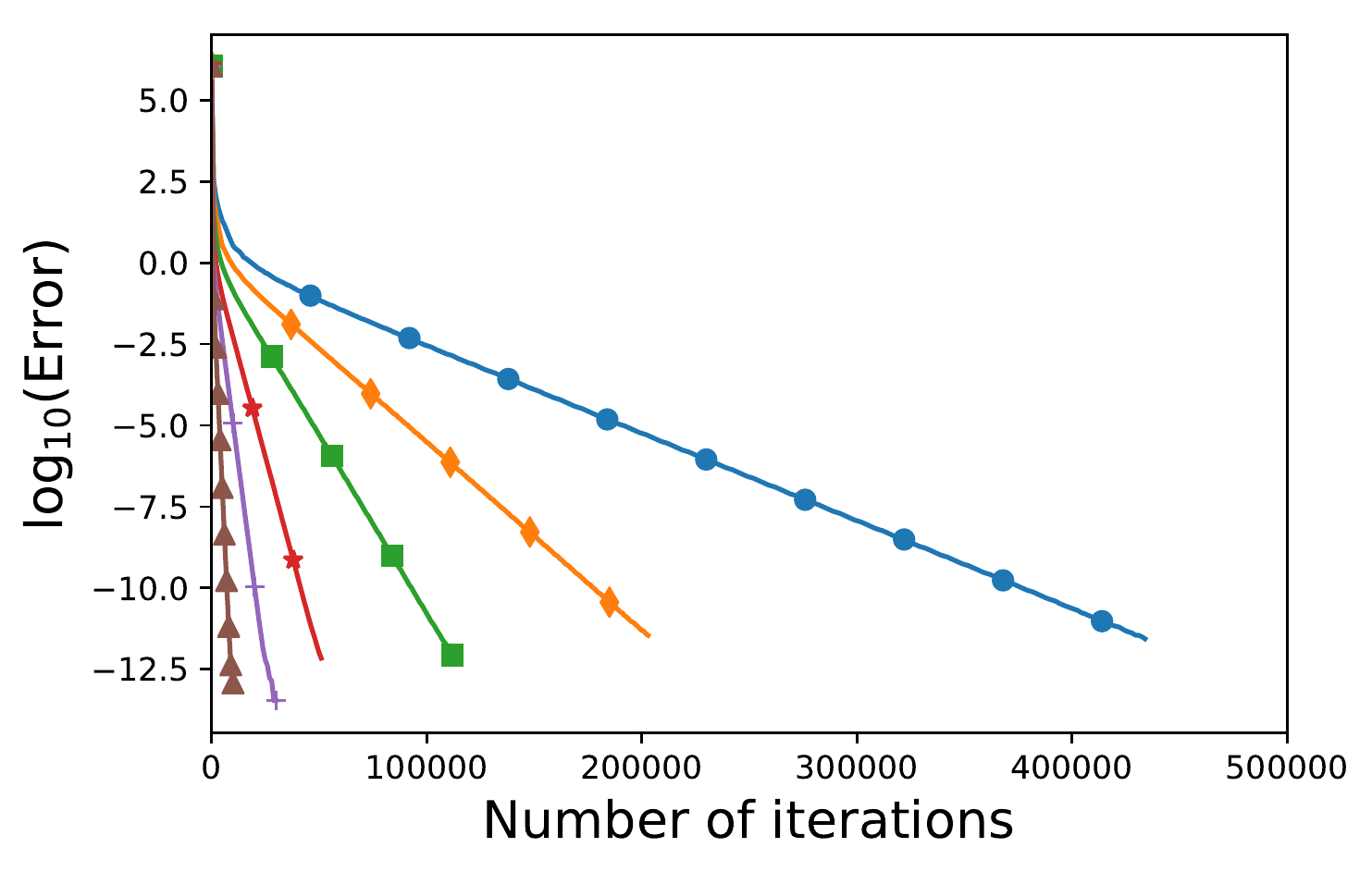}
    \caption{Quadratic CFPs  (iterations)}
\end{subfigure}
\begin{subfigure}{0.49\textwidth}
    \includegraphics[width=\textwidth]{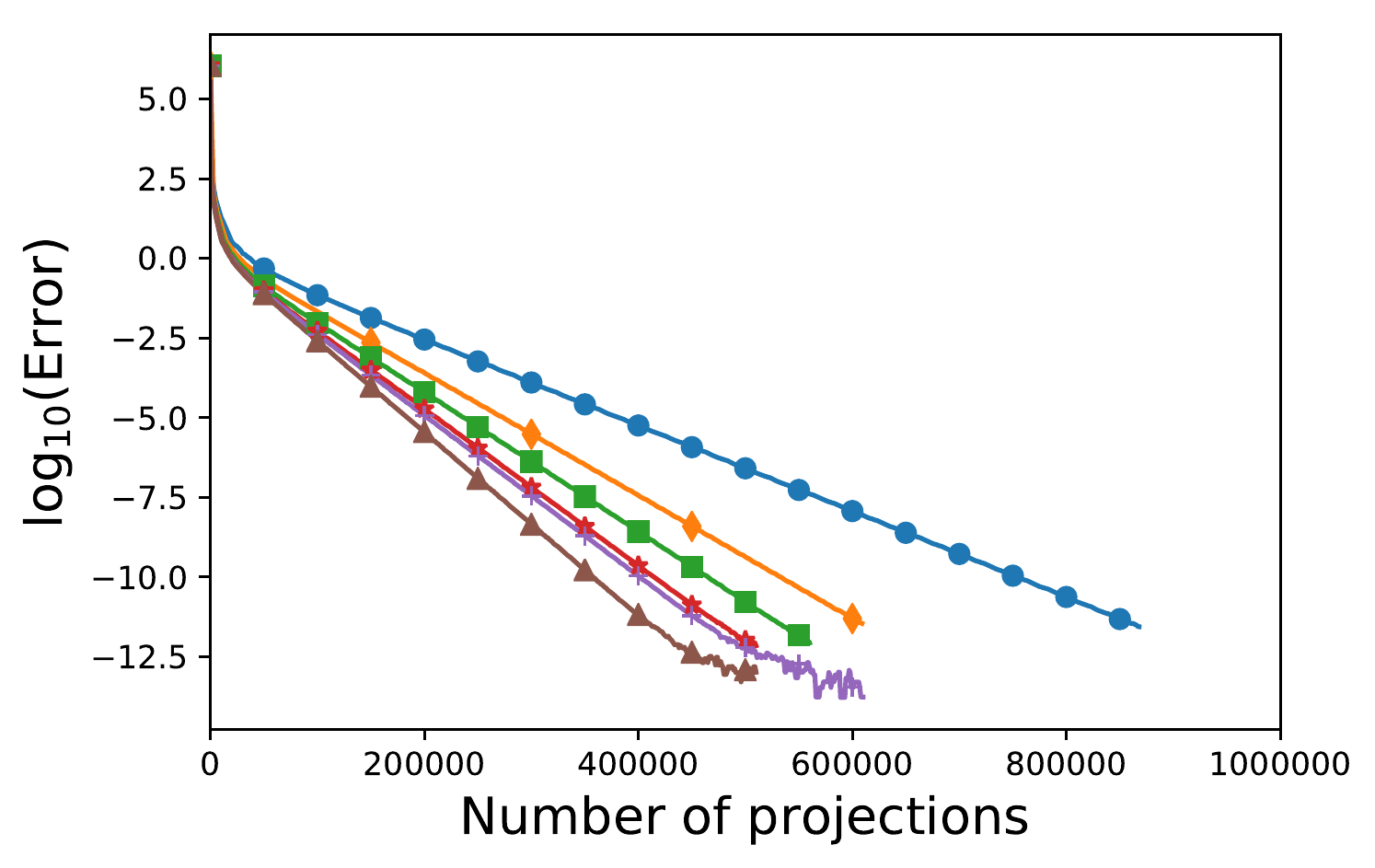}
    \caption{Quadratic CFPs (projections)}
\end{subfigure}
\caption{Comparison of the value of Error in~\eqref{eq:def_error} and the number of iterations and projections used for one randomly generated problem with 10,000 constraints. The original cyclic DR algorithm with $r=2$ needs more projections to achieve the same accuracy than the $r$-sets DR methods with $r>2$.}
\label{fig:dist}
\end{center}
\end{figure}

Extensive numerical study is called for, and indeed planned for future work, to explore further the computational aspects
of the of $r$-sets DR operators with $r>2$.
\bigskip

\textbf{Acknowledgments.} We thank Prof. Toufik Mansour for help in
formulating our ideas, Yehuda Zur for his meticulous work on initial Matlab experiments, and
Rafiq Mansour for enlightening comments about the manuscript.
We greatly appreciate the constructive comments of two anonymous reviewers which helped us improve the paper.
The first author was supported by MINECO of Spain and ERDF of EU, as part of the Ram\'on y Cajal
program (RYC-2013-13327) and the Grant MTM2014-59179-C2-1-P.
The second author's work was supported by research grant no. 2013003 of the United States-Israel Binational Science Foundation (BSF). The third author's work was supported by the EU FP7 IRSES program STREVCOMS, grant no. PIRSES-GA-2013-612669.

\end{document}